\documentclass[12pt]{article}%
\usepackage{amsmath,amssymb,amsfonts,amscd,amsthm}
\usepackage{amsfonts}
\usepackage{amsmath}
\usepackage{graphicx}%

\usepackage{tikz}
\usetikzlibrary{arrows,shapes,positioning}
\usetikzlibrary{decorations.markings}
\tikzstyle arrowstyle=[scale=1]
\tikzstyle directed=[postaction={decorate,decoration={markings,
    mark=at position .65 with {\arrow[arrowstyle]{stealth}}}}]
\tikzstyle reverse directed=[postaction={decorate,decoration={markings,
    mark=at position .65 with {\arrowreversed[arrowstyle]{stealth};}}}]
    
\providecommand{\U}[1]{\protect\rule{.1in}{.1in}}
\newtheorem{theorem}{Theorem}

\newtheorem{cor}[theorem]{Corollary}

\newtheorem{definition}[theorem]{Definition}

\newtheorem{rem}[theorem]{Remark}

\newcommand {\R} {\mathbb{R}}

\newcommand{\slt}{\mathfrak{sl}(2,\R)}

\begin{document}
  \begin{center}
        {\fontsize{18}{22}\selectfont
       \bf On the non-existence of hyperbolic polygonal relative equilibria for the negative curved 
       $n$--body problem with equal masses}
       \end{center}

\vspace{4mm}

        \begin{center}
        {\bf Ernesto P\'erez-Chavela$^1$, and Juan Manuel S\'anchez-Cerritos$^2$}\\
\bigskip
$^1$Departamento de Matem\'aticas\\
Instituto Tecnol\'ogico Aut\'onomo de M\'exico, Mexico City, Mexico\\
\bigskip
$^2$Departamento de Matem\'aticas\\
Universidad Aut\'onoma Metropolitana - Iztapalapa, Mexico City, Mexico\\
\bigskip
ernesto.perez@itam.mx, jmsc@xanum.uam.mx 
       \end{center}

        
        \begin{center}
        \today
        \end{center}

        
\abstract{We consider the $n$--body problem defined on surfaces of constant negative curvature. For the case of $n$--equal masses we prove that the hyperbolic relative equilibria with a regular polygonal shape do not exist. In particular the Lagrangian (three equal distances) hyperbolic relative equilibria do not exist. We also show the existence of a new class of hyperbolic collinear relative equilibria for the five body problem on surfaces of constant negative curvature.}
        

\section{Introduction}

We consider the gravitational $n$--body problem in the two-dimensional hyperbolic space. We use the formulation of the problem proposed by Diacu, P\'erez-Chavela and Santoprete which can be found in \cite{Diacu}. For an interesting historical background about this problem you can see \cite{Diacu4}.

Relative equilibria are particular solutions of the equations of motion where the distances among the particles remain constant along the motion. These solutions have been widely studied in the classical problem, and recently there have appeared some results for the curved case. For the case of positive curvature, the classical Newtonian Eulerian and Lagrangian relative equilibria has been generalized  in several papers, see for instance 
 \cite{Diacu,Zhu}. Some results about the stability of these solutions can be found in \cite{Martinez}, where we can see some differences with respect to the Newtonian case. An interesting point to emphasize is that in the positively-curved case the linear stability depends on the angular momentum.
Some authors have also studied the polygonal relative equilibria in the positive curved space \cite{Diacu5,Tibboel}.
We can see that as in the classical Newtonian $n$--body problem (zero curvature), $n$--equal masses located at the vertices of a regular $n$--gon generate a relative equilibria by taking the action of the $SO(2)$--group. In particular for $n=3$ it has been shown in \cite{Diacu}, that  three masses form a Lagrangian  relative equilibrium (three equal distances) iff they are equal.

When the curvature is negative, for case of the collinear relative equilibria, we have the same result for symmetric configurations 
(two equal masses at the ends with one arbitrary mass at the middle \cite{juan}). 

We can mention more articles for relative equilibria in the positive curvature space; however very few   of these trajectories for the negative curvature case have been studied. One of them is the paper written by Tibboel \cite{Tibboel}, where he shows the non existence of elliptic homographic orbits for irregular polygons with coordinate $z$ not constant (he considers positive and negative curvature).
For case of the collinear relative equilibria,the linear stability also depends on the angular momentum, in particular for symmetric configurations, two equal masses at the ends with one arbitrary mass at the middle, an extensive analysis of the linear stability can be found in \cite{juan}).

We know that, up to isometries, there are three different groups of isometry in surfaces of constant negative curvature, therefore in principle we can have three different classes of relative equilibria called elliptic, hyperbolic or parabolic (ahead in this paper we will go deeper in this point), however it is known that the parabolic relative equilibria do  not exists \cite{Diacu}. The elliptic relative equilibria are the solutions which are invariant  
under the action of the $SO(2)$--group. For the $n$--body problem with equal masses it has been proved that if the masses are located at the vertices of a regular $n$--gon, then they generate an relative equilibrium \cite{Diacu}. In this paper we tackle the same question for the hyperbolic relative equilibria: Are there hyperbolic relative equilibria with an initial configuration of regular $n$-gon?
For our surprise the answer is negative, in particular we show the non-existence of hyperbolic Lagrangian solutions. In section 3, after the preliminaries given in section 2, we prove our main result about the non-existence of relative equilibria with polygonal regular $n$-gon shape for the $n$--body problem with equal masses on spaces of negative curvature. In section 4 we introduce the concept of collinear relative equilibria for the $n$ body problem, that is relative equilibria where the masses are on the same geodesic. We prove the existence of new families of collinear hyperbolic relative equilibria for the $5$--body problem on spaces of negative curvature.

\section{Relative equilibria on surfaces of negative curvature}\label{RE}
Without loss of generality we consider surfaces of constant negative curvature $-1$, which in general is a complete, simply connected two dimensional space $\mathcal{H}^2$. There are several models from the hyperbolic geometry (all of them isometric) to represent $\mathcal{H}^2$. In this paper we will use two of them the Weierstrass model denoted by $\mathbb{L}^2$  and the Poincar\'e upper half plane model 
$\mathbb{H}^2$.

\subsection{The Weierstrass model}
The Weierstrass model also known as the pseudo sphere $\mathbb{L}^2$ is given by the upper sheet  of the hyperboloid
\[\mathbb{L}^2 =\{(x,y,z) \in \mathbb{R}^{2,1} : x^2+y^2-z^2=-1\}, \]
where $\mathbb{R}^{2,1}$ is the Minkowski space, that is $\mathbb{R}^3$ endowed with the Lorentz inner product denoted by $\odot$, given $a=(a_x,a_y,a_z)$ and $b=(b_x,b_y,b_z)$, we have 
$a \odot b = a_xb_x + a_yb_y - a_zb_z)$. Let us denote by $q_i$ the position of the particle with mass $m_i$. The distance between any two points in this space is $d(q_i,q_j)=\cosh^{-1}(-q_i\odot q_j)$.
The potential is given by
\[ U(q)=\sum_{i<j} \cot(d(q_i,q_j)). \]

And the kinetic energy is defined by

\[ T=\dfrac{1}{2}\sum_{i}m_i\dot{q}_i \odot \dot{q}_i. \]

From the Euler-Lagrange equations, the equations of motion take the form

\begin{equation} \ddot{q}_i=\sum_{i \neq j}  \dfrac{m_j(q_j + (q_i \odot q_j) q_i)}{(-1+(q_i \odot q_j)^2)^{3/2}}+(\dot{q}_i \odot \dot{q}_i)q_i, \ \ i=1,\cdots, n. \label{eq:motionL2}
 \end{equation}

Let $Lor(\mathbb{L}^2, \odot)$ be the group of all orthogonal transformations of determinant 1 that maintains the upper part of the hyperboloid invariant (the group of isometries of $\mathbb{L}^2$), see
\cite{Diacu,Guadalupe} for more details. Applying the Principal Axis Theorem \cite{Nomizu}, which states that any  
 $1$--parameter subgroup  of $Lor(\mathbb{L}^2, \odot)$ can be written, in a proper basis, as

\[A=P  \left( \begin{array}{ccc}
\cos \theta & -\sin \theta & 0 \\
\sin \theta & \cos \theta & 0 \\
0 & 0 & 1 \end{array} \right) P^{-1},\]

\[B=P  \left( \begin{array}{ccc}
1 & 0 & 0 \\
0 & \cosh s & \sinh s \\
0 & \sinh s & \cosh s \end{array} \right) P^{-1},\]
or

\[C= P  \left( \begin{array}{ccc}
1 & -t & t \\
t & 1-t^2/2 & t^2/2 \\
t & -t^2/2 & 1+t^2/2 \end{array} \right) P^{-1},\]
where $\theta \in [0,2 \pi),s,t \in \mathbb{R}$, and $P \in Lor(\mathbb{L}^2, \odot)$, we obtain that 
any isometry of $Lor(\mathbb{L}^2, \odot)$ can be written as a composition of some of the above  three transformations. These transformations are called elliptic, hyperbolic and parabolic transformations respectively. The relative equilibria on $\mathbb{L}^2$ are the solutions of (\ref{eq:motionL2}) which are invariant under some isometry of $Lor(\mathbb{L}^2, \odot)$.

\subsection{The Poincar\'e upper half plane model}.
The Poincar\'e upper half plan model is given by $\mathbb{H}^2=\{w \in \mathbb{C} : Im(w)>0\}$ with the Riemannian metric

\[ -ds^2=\dfrac{4}{(w-\bar{w})^2}dwd\bar{w} \].

The potential is given by

\[ U(w,\bar{w})=\sum_{i<j}m_km_j\dfrac{(\bar{w}_k-w_k)(\bar{w}_i-w_i)-2(|w_k|^2+|w_j|^2)}{T_{k,j}}, \]
where

\begin{equation*}
\begin{split}
 T_{k,j}=&\left(4(Re(w_k)-Re(w_j))^2\left[\left(Re(w_k)-Re(w_j)\right)^2+2(Im(w_k)^2+Im(w_j)^2)\right]\right.+\\
 &\left. 4(Im(w_k)^2-Im(w_j)^2)^2 \right)^{1/2},
\end{split}
\end{equation*}

and the kinetic energy is

\[  T=\sum_{i}\dfrac{2m_k|\dot{w}_k|^2}{(w-\bar{w})^2} .  \]

The equations of motion take the form
\begin{equation}\label{eq:motionH2}
 \ddot{w}_k=-\dfrac{(w_k-\bar{w}_k)^2}{2}\sum_{i\neq j}\dfrac{m_j(\bar{w}_k-w_k)(\bar{w}_j-w_j)^2(w_k-w_j)(\bar{w}_j-w_k)}{T_{j,k}^3}+\dfrac{2\dot{w}_k^2}{w_k-\bar{w}_k} \end{equation}

Let be
\[ {\rm SL}(2,\mathbb{R}) = \{ g \in {\rm GL}(2,\mathbb{R}) \, | \, \det g=1 \},  \]
where ${\rm GL}(2,\mathbb{R})$ the group of invertible $2 \times 2$ real matrices.

The group $ {\rm SL}(2,\mathbb{R}) $ defines an action $\Psi$ on $\mathbb{H}^2$ as follows. 
Given a matrix $g \in {\rm SL}(2,\mathbb{R})$, 
\begin{equation*}
g = \left ( \begin{array}{cc} a & b \\ c & d \end{array} \right ),
\end{equation*}
and an element $x+iy:=(x,y)\in \mathbb{H}^2$,  we define $\Psi(g,(x,y))$ via  the  Moebius transformation
\begin{equation*}
\Psi(g,(x+iy))= \frac{a(x+iy) + b}{c(x+iy) +d}.
\end{equation*}

It is well known  that  for any $g\in{\rm SL}(2,\mathbb{R}) $,  the map $\Psi(g,\cdot)$ is a Riemannian isometry of 
$\mathbb{H}^2$ equipped with the hyperbolic metric. Moreover, the group of proper isometries of $\mathbb{H}^2$ is the quotient
group $\displaystyle {\rm SL}(2,\mathbb{R})/\{\pm I \}:={\rm PSL}(2,\mathbb{R})$. The quotient is taken to account for the  fact that $\Psi(g,\cdot)=\Psi(-g,\cdot)$. 

The  Lie algebra of ${\rm SL}(2,\mathbb{R})$ is the 3-dimensional real linear space 
\[ {\mathfrak sl}(2,\mathbb{R}) = \{ \xi \in {\rm M}(2,\mathbb{R}) \, | \, \,  {\rm trace\  \! \xi}=0  \}.  \]
Hence, any non-zero
element in $\slt$ is of one of the following three types:

\begin{enumerate}
\item Elliptic. These elements have two complex conjugate, purely imaginary eigenvalues.
\item Hyperbolic. They posses two real eigenvalues with the same absolute value and opposite signs.
\item Parabolic. They have a multiplicity two zero eigenvalue  (and are not diagonalizable). 
\end{enumerate}

The following are canonical representatives of elements of the above types. 
\begin{equation*}
\xi_e:=\left ( \begin{array}{cc} 0 & -\frac{1}{2} \\ \frac{1}{2} & 0 \end{array} \right ), \qquad
\xi_h:=\left ( \begin{array}{cc} \frac{1}{2} &0 \\ 0 & -\frac{1}{2} \end{array} \right ), \qquad \xi_p^{\pm}:=\left ( \begin{array}{cc} 0 &\pm 1 \\ 0 & 0 \end{array} \right ).
\end{equation*}

An arbitrary elliptic (respectively, hyperbolic) element
in $\slt$ is equivalent to $\omega \xi_e$ (respectively, $\omega \xi_h$) for a certain $\omega >0$   by ${\rm SL}(2,\R)$-conjugation.
Parabolic elements are either equivalent to $\xi_p^+$ or $\xi_p^-$ by ${\rm SL}(2,\R)$-conjugation.

Canonical representatives of the orbits are
$\omega \xi_e, \, \omega \xi_h, \xi_p^+$ and $\xi_p^-$, where $\omega$ is a positive real parameter.

The search for relative equilibria in the negative curved $n$--body problem is equivalent to find  a configuration $q_0 \in \mathbb{H}^{2n}$ such that 
\begin{equation*}
\Psi (\exp(\xi_a t),q_0) 
\end{equation*}
is a solution of the equations of motion, where the action is acting coordinate to coordinate for $a=e,h,p$.

\section{Hyperbolic Relative Equilibria}
According with Section \ref{RE} we can have essentially three different kinds of relative equilibria on surfaces of constant negative curvature. It is known that the parabolic relative equilibria do not exist \cite{Diacu}. The elliptic relative equilibria has been studied by several authors, as we have mentioned in the introduction, but very few is known about the hyperbolic relative equilibria. In the literature appears only some results for collinear solutions (the three particles are on the same geodesic) see for instance \cite{Diacu}, \cite{juan} and  . In \cite{naranjo} the authors have done a complete analysis of the relative equilibria and their stability but just for the case $n=2$.

In this paper we are interested in the hyperbolic relative equilibria, that is, in motions generated by the hyperbolic transformations. Our main goal is to prove the non-existence of relative equilibria with regular $n$--gon shape.

\begin{definition}
A hyperbolic circle with center $q_0$ and radius $\rho$ is the set of points on $\mathcal{H}^2$ whose hyperbolic distance to $q_0$ is $\rho$.
\end{definition}

\begin{definition}
We say that $n$ points $\{ a_i \}$ in $\mathcal{H}^2$ form a regular $n$--gon if they are on a hyperbolic circle and all hyperbolic distances $a_ia_{i+1}$ are equal for $i=1,,\dots, n$, where $a_{n+1}=a_1$.
\end{definition}

\begin{theorem}
In the $n$--body problem on $\mathcal{H}^2$ with equal masses do not exist hyperbolic relative equilibria with a regular $n$--gon configuration
\end{theorem}

\begin{proof}
For the proof of this theorem we will use the Weierstrass model $\mathbb{L}^2$.
Consider an initial configuration $q(0)=(q_1(0),\cdots, q_n(0)) \in (\mathbb{L}^2)^n$ with a regular 
$n$--gon shape $\Delta$. For any regular $n$--gon there exist an isometry that transforms the center 
of the polygon $\Delta$ into the point $(0,0,1)$,  and the hyperbolic geodesic that connects $a_1$ with the center of $\Delta$ into the geodesic $y^2-z^2=-1$ $(x=0)$.

This new initial configuration in the same class that the original one can be written as

\begin{equation}
\begin{split}
q_i(0)= & \left(r\cos \left(\dfrac{2 \pi(i-1)}{n}+ \dfrac{\pi}{2}\right),r \sin\left(\dfrac{2 \pi(i-1)}{n}+ \dfrac{\pi}{2}\right),z \right) \\
& \left(-r\sin \left(\dfrac{2 \pi(i-1)}{n}\right),r \cos\left(\dfrac{2 \pi(i-1)}{n}\right),z \right),\quad  i=1,\cdots n,
\end{split}
\end{equation}
with $r$ and $z$ are fixed and satisfy $r^2-z^2=-1$.

By section \ref{RE}, we know that the hyperbolic transformations are given by the Lorentz transformations 

\begin{equation}
A(s)=\begin{bmatrix}
1 &   0 & 0  \\[0.3em]
0  & \cosh(s) & \sinh(s) \\[0.3em]
0           & \sinh(s)           & \cosh(s)
\end{bmatrix}.
\end{equation}

The hyperbolic relative equilibria  correspond to a Lorentz transformation with a particular ``angular velocity'' $\omega\neq 0$ applied to a given initial configuration \cite{naranjo}. Hence if we apply  $A(\omega t)$ to each point  $q_i(0)$ we get

\begin{equation}
q_i(0) \rightarrow  q_i(t)= \begin{bmatrix} -r\sin \left(\dfrac{2 \pi(i-1)}{n}\right) \\ r\cosh(\omega t)  \cos\left(\dfrac{2 \pi(i-1)}{n}\right)  + z \sinh(\omega t) \\  r\sinh(\omega t)  \cos\left(\dfrac{2 \pi(i-1)}{n}\right)  + z \cosh(\omega t)\end{bmatrix}.
\end{equation}

Consider the following expressions
\begin{equation*}
(\dot{q}_i \odot \dot{q}_i)=-\omega^2 \left( \cos^2\left( \dfrac{2\pi(i-1)}{n} \right)r^2-z^2  \right),
\end{equation*} 
and
\begin{equation*} \ddot{z}_i=\omega^2 z_i. \end{equation*}
We have
\begin{equation}
\ddot{z_1}=\sum_{j=2}^{n}\dfrac{m_j[z_j+(q_1 \odot q_j)z_1]}{[-1+(q_1 \odot q_j)^2]^{3/2}}+(\dot{q}_1 \odot \dot{q}_1)z_1.
\label{movz}
\end{equation}
Now, using the particular case
$$(\dot{q}_1 \odot \dot{q}_1)=-\omega^2 \left(r^2-z^2  \right), \quad {\rm and} \quad 
\ddot{z}_1=\omega^2 z_1 .$$
We obtain  
$$(\dot{q}_1 \odot \dot{q}_1)z_1-\ddot{z}_1=-\omega^2(r^2-z^2+1)z_1=0.$$

From this fact and equation (\ref{movz}), we  get

\begin{equation}
\sum_{j=2}^{n}\dfrac{m_j[z_j+(q_1 \odot q_j)z_1]}{[-1+(q_1 \odot q_j)^2]^{3/2}}=0. \label{zzero}
\end{equation}

Notice that

$$q_i \odot q_j=\cos^2\left( \dfrac{2\pi(i-j)}{n} \right)r^2-z^2 \leq r^2-z^2=-1.$$

Then 
$$z_j+(q_1 \odot q_j)z_1< z_j-z_1, \ \ \ j \neq 1. $$

If we consider $\omega>0$ and $t>0$, then we have
\begin{eqnarray*}
z_1 &=& r\sinh(\omega t)  + z \cosh(\omega t) \\
    &>& r\sinh(\omega t)  \cos\left(\dfrac{2 \pi(j-1)}{n}\right)  + z \cosh(\omega t)\\
    &=& z_j,  \quad \forall \quad j \neq 1.
\end{eqnarray*}

All the above  imply that  
$$z_j+(q_1 \odot q_j)z_1< z_j-z_1<0, \ \ \ \forall j \neq 1.$$

This fact indicates that equation $(\ref{zzero})$ is never satisfied.

If we consider $\omega<0$, then we obtain the same result by considering $t<0$.

Therefore hyperbolic relative equilibria do not exist with a regular $n$--gon shape.
\end{proof}

As an easy consequence of the above theorem we obtain the following result for the three body problem
concerning to the non-existence of hyperbolic Lagrangian relative equilibria (the hyperbolic distance among the three masses is the same). Actually we prove first this result and then we discover that it can be extended to the $n$--body problem with equal masses.

\begin{cor}
There are not Lagrangian hyperbolic relative equilibria
\end{cor}

\section{A new class of hyperbolic collinear relative equilibria for the $5$--body problem}

Hyperbolic collinear solutions for the three body problem on spaces of constant negative curvature, when the two masses at the end of the same geodesic are equal and the mass in the middle is arbitrary where found by Diacu, P\'erez-Chavela and Santoprete in \cite{Diacu}, hence, a natural question is whether or not these kind of solutions exist on the case for more bodies. In this section we prove the following result for the $5$--body problem.

\begin{theorem}\label{collinear}
In the 5-body problem on $\mathcal{H}^2$ we consider $5$--particles on the same geodesic with masses
$m_1=m_2= m, m_3=M$ and $m_4=m_5=\mu$, such that the couple of equal masses are equidistant with $M$. Then, there exist positions and masses such that they lead to collinear hyperbolic relative equilibria.
\end{theorem}

\begin{proof}
For the proof of this theorem we will use as a model for $\mathcal{H}^2$
the Poincar\'e upper half plane  of the hyperbolic geometry $\mathbb{H}^2$.

Consider five particles with an initial position on the same geodesic. It is possible to consider this geodesic as the canonical one, that is as the half circle of  Euclidean radius one with center at the point $(0,0)$ (see \cite{naranjo} for more details).

We assume that the initial positions of these particles are	 $q_i(0)=(\cos(\theta_i), \sin(\theta_i))$. The transformation $\Psi$ which defines ``hyperbolic rotations" (in this case it correspond to homothetic motions) is given by $\Psi(q_i)=e^{\omega t}q_i$.

As we have seen in Section \ref{RE}, in order to obtain hyperbolic relative equilibria, we must guarantee the existence of a constant $\omega$ such that $q_j(t)=e^{\omega t}q_j(0)$ for 
$j=1,\cdots,n$ is a solution of the equations of motion (\ref{eq:motionH2}).

We consider the following initial positions 

\begin{equation}
\begin{split}
q_1(0)=&\left(\cos\left(\dfrac{\pi}{2}-\alpha\right),\sin\left(\dfrac{\pi}{2}-\alpha\right)\right)\\
q_2(0)=&\left(-\cos\left(\dfrac{\pi}{2}-\alpha\right),\sin\left(\dfrac{\pi}{2}-\alpha\right)\right)\\
q_3(0)=&(0,1)\\
q_4(0)=&\left(\cos\left(\dfrac{\pi}{2}-\alpha-\beta\right),\sin\left(\dfrac{\pi}{2}-\alpha-\beta\right)\right)\\
q_5(0)=&\left(-\cos\left(\dfrac{\pi}{2}-\alpha-\beta\right),\sin\left(\dfrac{\pi}{2}-\alpha-\beta\right)\right),
\end{split}
\end{equation}
for the particles with masses $m_1=m_2=\mu, \ \ m_3=M, \ \ m_4=m_5=m$.

For any $\alpha \in \left(0, \dfrac{\pi}{2}\right)$ we will show that there exist $\beta \in \left(0, \dfrac{\pi}{2}\right)$ with $\alpha + \beta<\dfrac{\pi}{2}$ and values of $m$ and $M$ such that they generate hyperbolic relative equilibria.

Substituting the curves $q_i(t), \ \ i=1,\cdots 5$ into the equations of motion (\ref{eq:motionH2}) we obtain the following system

\begin{equation}
\begin{split}
\omega_1^2=&-\,\frac { \left( \cos \left( \alpha \right)  \right) ^{4}}{\sin \left( 
\alpha \right) } \left( -2\mu\,{\frac {\sin \left( \alpha \right)  \left( \cos
 \left( \alpha \right)  \right) ^{2}}{ \left( 1+ \left( \sin \left( \alpha
 \right)  \right) ^{2} \right) ^{3}}}-M\sin \left( \alpha \right)+\right.\\ 
 &\left.{\frac {m \left( \cos \left( \alpha+\beta \right)  \right) ^{2} \left( 
\sin \left( \alpha+\beta \right) -\sin \left( \alpha \right)  \right) }{\, \left( 
1-\sin \left( \alpha \right) \sin \left( \alpha+\beta \right)  \right) ^{3}}}-{
\frac {m \left( \cos \left( \alpha+\beta \right)  \right) ^{2} \left( \sin
 \left( \alpha+\beta \right) +\sin \left( \alpha \right)  \right) }{\, \left( \sin
 \left( \alpha \right) \sin \left( \alpha+\beta \right) +1 \right) ^{3}}} \right)\label{w1},
 \end{split}
 \end{equation}
 \begin{equation}
\begin{split}
 \omega_2^2=&-\,\frac { \left( \cos \left( \alpha+\beta \right)  \right) ^{4}}{\sin
 \left( \alpha+\beta \right) } \left( -\mu{\frac { \left( \cos \left( a \right) 
 \right) ^{2} \left( \sin \left( \alpha+\beta \right) -\sin \left( \alpha \right) 
 \right) }{\, \left( 1-\sin \left( \alpha \right) \sin \left( \alpha+\beta
 \right)  \right) ^{3}}}\right.  \\
 &\left.-\mu{\frac { \left( \cos \left( \alpha \right) 
 \right) ^{2} \left( \sin \left( \alpha+\beta \right) +\sin \left( \alpha \right) 
 \right) }{\, \left( \sin \left( \alpha \right) \sin \left( \alpha+\beta \right) +
1 \right) ^{3}}}-M\sin \left( \alpha+\beta \right)-2\,{\frac 
{m \left( \cos \left( \alpha+\beta \right)  \right) ^{2}\sin \left( \alpha+\beta
 \right) }{ \left(  \left( \sin \left( \alpha+\beta \right)  \right) ^{2}+1
 \right) ^{3}}} \right) . \label{w2}
\end{split}
\end{equation}

From equations (\ref{w1}) and (\ref{w2}) we get

\begin{equation}
\begin{split}
\omega_1^2-\omega_2^2 =&  2 \left( \dfrac{\mu \cos^6(\alpha)}{(1+\sin^2(\alpha))^3}-\dfrac{m\cos^6(\alpha+\beta)}{(1+\sin^2(\alpha+\beta))^3} \right)+M\left(\cos^4(\alpha)-\cos^4(\alpha+\beta)\right)\\
&-\dfrac{\cos^2(\alpha+\beta)\cos^2(\alpha)(\sin(\alpha+\beta)-\sin(\alpha))}{(1-\sin(\alpha)\sin(\alpha-\beta))^3}\left( \dfrac{m\cos^2(\alpha)}{\sin(\alpha)}+\dfrac{\mu \cos^2(\alpha+\beta)}{\sin(\alpha+\beta)} \right)\\
&+\dfrac{\cos^2(\alpha+\beta)\cos^2(\alpha)(\sin(\alpha+\beta)+\sin(\alpha))}{(1+\sin(\alpha)\sin(\alpha-\beta))^3}\left( \dfrac{m\cos^2(\alpha)}{\sin(\alpha)}-\dfrac{\mu \cos^2(\alpha+\beta)}{\sin(\alpha+\beta)} \right)\\
&=M\left(\cos^4(\alpha)-\cos^4(\alpha+\beta)\right)\\
&+m \left(   -\dfrac{\cos^2(\alpha+\beta)\cos^2(\alpha)(\sin(\alpha+\beta)-\sin(\alpha))}{(1-\sin(\alpha)\sin(\alpha-\beta))^3} \dfrac{\cos^2(\alpha)}{\sin(\alpha)} \right.\\
&\left. +\dfrac{\cos^2(\alpha+\beta)\cos^2(\alpha)(\sin(\alpha+\beta)+\sin(\alpha))}{(1+\sin(\alpha)\sin(\alpha-\beta))^3}\dfrac{\cos^2(\alpha)}{\sin(\alpha)}  \right.\\
&\left. -\dfrac{2\cos^6(\alpha+\beta)}{(1+\sin^2(\alpha+\beta))^3}  \right)\\
&+\mu\left(-\dfrac{\cos^2(\alpha+\beta)\cos^2(\alpha)(\sin(\alpha+\beta)-\sin(\alpha))}{(1-\sin(\alpha)\sin(\alpha-\beta))^3}\dfrac{\cos^2(\alpha+\beta)}{\sin(\alpha+\beta)}\right.\\
&-\dfrac{\cos^2(\alpha+\beta)\cos^2(\alpha)(\sin(\alpha+\beta)+\sin(\alpha))}{(1+\sin(\alpha)\sin(\alpha-\beta))^3}\dfrac{\cos^2(\alpha+\beta)}{\sin(\alpha+\beta)}\\
&\left.+  \dfrac{2\cos^6(\alpha)}{(1+\sin^2(\alpha))^3}\right) . \label{h}
\end{split}
\end{equation}

The above equation can be seen as

\begin{equation}
\omega_1^2-\omega_2^2=f_1(\alpha,\beta)M+f_2(\alpha,\beta)m+f_3(\alpha,\beta)\mu. \label{eq:resta}
\end{equation}

We will see that for some $\alpha, \beta$ with $0< \alpha < \alpha + \beta<\dfrac{\pi}{2}$, there exist positive values for the masses such that the right side of (\ref{eq:resta}) is zero. In other words, we will see that there are $\alpha,\beta, m, M$ and $\mu$ such that  $\omega_1^2=\omega_2^2$. These are the values of $\omega = |\omega_1|=|\omega_2|$ that we are looking for get a relative equilibrium.

It is easy to verify that $f_1>0$, since $\alpha \in (0,\frac{\pi}{2})$ and $\beta \in (0,\frac{\pi}{2}-\alpha)$.

Consider the line $\beta=\frac{\pi}{2}-\alpha$, then

\[ f_2(\alpha,\frac{\pi}{2}-\alpha)=\frac{P}{Q},
 \]
where

\[ P=-(4\, \left( \cos \left( a \right)  \right) ^{6}-4\, \left( \cos \left( 
a \right)  \right) ^{4}-3\, \left( \cos \left( a \right)  \right) ^{2}
+1
)
 \]
 
 \[ Q=\left( 4\, \left( \cos \left( a \right)  \right) ^{4}-8\, \left( \cos
 \left( a \right)  \right) ^{2}+5 \right) ^{3}.
 \]
 
 Let us see that $P$ has a root in the interval $(0,\frac{\pi}{2})$. To verify this consider $x=\cos^2(\alpha)$, then the polynomial $P$ transforms to $\bar{P}=-(4x^3-4x^2-3x+1)$. This polynomial has two critical points $x_{1,2}=\frac{1}{3}\pm\frac{1}{6}\sqrt{13}$, one positive corresponding to a maximum and one negative corresponding to a minimum. Since 
 $\lim_{x \to \infty} = - \infty$ and $\bar{P}(1)=2$ we conclude that $\bar{P}$ has only one root $x_0 \in (0,1)$. This implies that there exist only one value $\alpha_0 \in (0,\frac{\pi}{2})$ such that $P(\alpha_0)=0$
 
 Consider $P(0)=2$, $P(\frac{\pi}{2})=-1$, then there exists only one value $\alpha_1 \in (0,\frac{\pi}{2})$ such that $P(\alpha_1)=0$. For $\alpha \in (\alpha_1,\frac{\pi}{2})$ we have $f_2(\alpha,\frac{\pi}{2}-\alpha)<0$. By continuity of the function $f_2(\alpha,\beta)$, there exists values of $\alpha \in (0, \frac{\pi}{2})$ and $\beta \in (0, \frac{\pi}{2}-\alpha)$ such that $f_2(\alpha,\beta)<0$. For those values of $\alpha$ and $\beta$ for which $f_2<0$, and since $f_1>0$ we can conclude that, independently of the value of $f_3$, it is possible to find values of $m$, $M$ and $\mu$ such that the right side of (\ref{eq:resta}) is zero. With this we conclude the proof of the theorem. \end{proof}

\begin{rem} We believe that if we add another couple of equal masses equidistant with the particle of mass $M$ (and so on), it will be possible  to obtain a similar result to Theorem \ref{collinear}. However the computations are much more complicated.
\end{rem}

\begin{rem}
Numerically we can see in Figure \ref{region} the region where $f_2<0$ for $\alpha \in (0,\frac{\pi}{2}), \beta \in (0,\frac{\pi}{2}-\alpha)$. Also numerically it is possible to check that $f_3>0$ for any $\alpha \in (0,\frac{\pi}{2}), \beta \in (0,\frac{\pi}{2}-\alpha)$.
\end{rem}
\begin{figure}[h]
\centering
{\includegraphics[width=.5\textwidth]{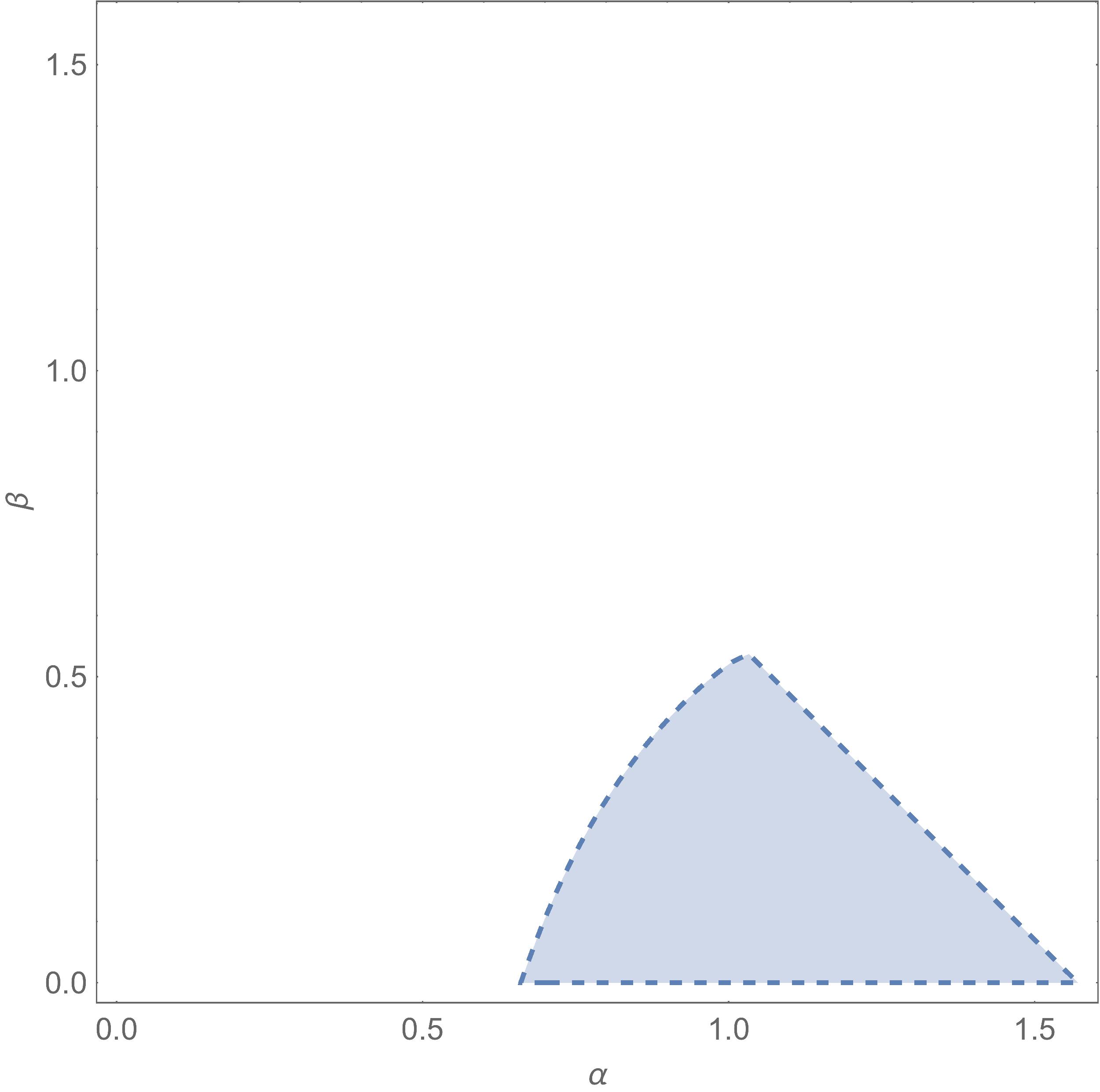}}
\caption{Region where $f_2(\alpha,\beta)<0$.}\label{region}
\end{figure}

\end{document}